\documentclass[oneside,english]{amsart}
\usepackage[T1]{fontenc}
\usepackage[latin9]{inputenc}
\usepackage{babel}
\usepackage{refstyle}
\usepackage{amsthm}
\usepackage{amstext}
\usepackage{amssymb}
\usepackage[unicode=true,pdfusetitle,
 bookmarks=true,bookmarksnumbered=false,bookmarksopen=false,
 breaklinks=false,pdfborder={0 0 1},backref=false,colorlinks=false]
 {hyperref}
\usepackage{breakurl}

\makeatletter

\AtBeginDocument{\providecommand\thmref[1]{\ref{thm:#1}}}
\AtBeginDocument{\providecommand\lemref[1]{\ref{lem:#1}}}
\AtBeginDocument{\providecommand\remref[1]{\ref{rem:#1}}}
\RS@ifundefined{subref}
  {\def\RSsubtxt{section~}\newref{sub}{name = \RSsubtxt}}
  {}
\RS@ifundefined{thmref}
  {\def\RSthmtxt{theorem~}\newref{thm}{name = \RSthmtxt}}
  {}
\RS@ifundefined{lemref}
  {\def\RSlemtxt{lemma~}\newref{lem}{name = \RSlemtxt}}
  {}

\numberwithin{equation}{section}
\numberwithin{figure}{section}
\theoremstyle{plain}
\newtheorem{thm}{\protect\theoremname}
  \theoremstyle{definition}
  \newtheorem{defn}[thm]{\protect\definitionname}
  \theoremstyle{remark}
  \newtheorem{rem}[thm]{\protect\remarkname}
  \theoremstyle{plain}
  \newtheorem{lem}[thm]{\protect\lemmaname}
  \theoremstyle{plain}
  \newtheorem{cor}[thm]{\protect\corollaryname}

\makeatother

  \providecommand{\corollaryname}{Corollary}
  \providecommand{\definitionname}{Definition}
  \providecommand{\lemmaname}{Lemma}
  \providecommand{\remarkname}{Remark}
\providecommand{\theoremname}{Theorem}

\begin{document}

\title{On a Fixed Point Theorem for a Cyclical Kannan-type Mapping}

\author{Mitropam Chakraborty }

\address{Mitropam Chakraborty\\
Department of Mathematics\\
Visva-Bharati\\
Santiniketan 731235 \\
India}

\email{mitropam@gmail.com}

\thanks{The first author is indebted to the \textbf{UGC} (University Grants
Commissions), India for awarding him a \textbf{JRF }(Junior Research
Fellowship) during the tenure in which this paper was written.}

\author{S. K. Samanta}

\address{S. K. Samanta\\
Department of Mathematics\\
Visva-Bharati\\
Santiniketan 731235 \\
India}

\email{syamal\_123@yahoo.co.in}
\begin{abstract}
This paper deals with an extension of a recent result by the authors
generalizing Kannan's fixed point theorem based on a theorem of Vittorino
Pata. The generalization takes place via a cyclical condition.
\end{abstract}



\maketitle

\section{Introduction}

Somewhat in parallel with the renowned Banach contraction principle
(see, for instance, \cite{granas2003fixed}), Kannan's fixed point
theorem has carved out a niche for itself in fixed point theory since
its inception in 1969 \cite{kannan1969some}. Let $(X,\, d)$ be a
metric space. If we define $T\colon X\to X$ to be a \emph{Kannan
mapping} provided there exists some $\lambda\in(0,\,1)$ such that
\begin{equation}
d(Tx,\, Ty)\leq\frac{\lambda}{2}[d(x,\, Tx)+d(y,\, Ty)]\label{eq:kannan}
\end{equation}
for each $x,\, y\in X$, then Kannan's theorem essentially states
that:

Every Kannan mapping in a complete metric space has a unique fixed
point. 

To see that the two results are independent of each other, one can
turn to \cite{rhoades1977comparison}, e.g., and Subrahmanyam has
shown in \cite{subrahmanyam1975completeness} that Kannan's theorem
characterizes metric completeness, i.e.: if every Kannan mapping on
a metric has a fixed point, then that space must necessarily be complete.

Kirk et al. \cite{kirk2003fixed} introduced the so-called cyclical
contractive conditions to generalize Banach's fixed point theorem
and some other fundamental results in fixed point theory. Further
works in this aspect, viz. the cyclic representation of a complete
metric space with respect to a map, have been carried out in \cite{rus2005cyclic,puacurar2010fixed,Karapinar2011822}.
Pata in \cite{Pata2011}, however, extended Banach's result in a totally
different direction and ended up proving that if $(X,\, d)$ is a
complete metric space and $T\colon X\to X$ a map such that there
exist fixed constants $\Lambda\geq0$, $\alpha\ge1$, and $\beta\in[0,\,\alpha]$
with 
\begin{equation}
d(Tx,\, Ty)\leq(1-\varepsilon)d(x,\, y)+\Lambda\varepsilon^{\alpha}\psi(\varepsilon)[1+\left\Vert x\right\Vert +\left\Vert y\right\Vert ]^{\beta}\label{eq:Pata}
\end{equation}
for every $\varepsilon\in[0,\,1]$ and every $x,\, y\in X$ (where
$\psi\colon[0,\,1]\to[0,\,\infty]$ is an increasing function that
vanishes with continuity at zero and $\left\Vert x\right\Vert \colon=d(x,\, x_{0}),\,\forall x\in X$,
with an arbitrarily chosen $x_{0}\in X$), then $T$ has a unique
fixed point in $X$. Combining Pata's theorem and the cyclical framework,
Alghamdi et al. have next come up with a theorem of their own \cite{Alghamdi2012}. 

On the one hand, proofs of cyclic versions of Kannan's theorem were
given in \cite{rus2005cyclic} and \cite{petric2010fixed}; the present
authors, on the other hand, have already established an analogue of
Pata's result that generalizes Kannan's theorem instead \cite{Chakraborty2012first}.
Letting everything else denote the same as in \cite{Pata2011} except
for fixing a slightly more general $\beta\geq0$, we have actually
shown the following: 
\begin{thm}
\cite{Chakraborty2012first}\label{thm:our-1st} If the inequality
\begin{equation}
d(Tx,\, Ty)\leq\frac{1-\varepsilon}{2}[d(x,\, Tx)+d(y,\, Ty]+\Lambda\varepsilon^{\alpha}\psi(\varepsilon)[1+\left\Vert x\right\Vert +\left\Vert Tx\right\Vert +\left\Vert y\right\Vert +\left\Vert Ty\right\Vert ]^{\beta}\label{eq:our-1st}
\end{equation}
is satisfied $\forall\varepsilon\in[0,\,1]$ and $\forall x,\, y\in X$,
then $T$ possesses a unique fixed point 
\[
x^{*}=Tx^{*}\,(x^{*}\in X).
\]

\end{thm}
In this article, we want to utilize \thmref{our-1st} to bridge the
gap by providing the only remaining missing link, i.e. a fixed point
theorem for cyclical contractive mappings in the sense of both Kannan
and Pata.

\section{The Main Result}

Let us start by recalling a definition which has its roots in \cite{kirk2003fixed};
we shall make use of a succinct version of this as furnished in \cite{Karapinar2011822}:
\begin{defn}
\label{defn:cyclic}\cite{Karapinar2011822} Let $X$ be a non-empty
set, $m\in\mathbb{N}$, and $T\colon X\to X$ a map. Then we say that
$\bigcup_{i=1}^{m}A_{i}$ (where $\emptyset\ne A_{i}\subset X\,\forall i\in\{1,\,2,\,\ldots\,,\, m\}$)
is a \emph{cyclic representation} of $X$ with respect to $T$ \emph{iff}
the following two conditions hold.\end{defn}
\begin{enumerate}
\item $X=\bigcup_{i=1}^{m}A_{i}$;
\item $T(A_{i})\subset A_{i+1}$ for $1\leq i\leq m-1$, and $T(A_{m})\subset A_{1}$.
\end{enumerate}
Now, let $(X,\, d)$ be a complete metric space. We have to first
assign $\psi\colon[0,\,1]\to[0,\,\infty]$ to be an increasing function
that vanishes with continuity at zero. With this, we are ready to
formulate our main result, viz.:
\begin{thm}
\label{thm:main}Let $\Lambda\geq0$, $\alpha\ge1$, and $\beta\geq0$
be fixed constants. If $A_{1},\,\ldots\,,\, A_{m}$ are non-empty
closed subsets of $X$ with $Y=\bigcup_{i=1}^{m}A_{i}$, and if $T\colon Y\to Y$
is such a map that $\bigcup_{i=1}^{m}A_{i}$ is a cyclic representation
of $Y$ with respect to $T$, then, provided the inequality 
\begin{equation}
d(Tx,\, Ty)\le\frac{1-\varepsilon}{2}[d(x,\, Tx)+d(y,\, Ty)+\Lambda\varepsilon^{\alpha}\psi(\varepsilon)[1+\left\Vert x\right\Vert +\left\Vert Tx\right\Vert +\left\Vert y\right\Vert +\left\Vert Ty\right\Vert ]^{\beta}\label{eq:us}
\end{equation}
is satisfied $\forall\varepsilon\in[0,\,1]$ and $\forall x\in A_{i},\, y\in A_{i+1}$
(where $A_{m+1}=A_{1}$ and, as in \cite{Pata2011}, $\left\Vert x\right\Vert \colon=d(x,\, x_{1}),\,\forall x\in Y$,
for an arbitrarily chosen $x_{1}\in Y$\textemdash{}a sort of ``zero''
of the space $Y$ ), $T$ has a unique fixed point $x^{*}\in\bigcap_{i=1}^{m}A_{i}$.\end{thm}
\begin{rem}
\label{rem:Since-we-can}Since we can always redefine $\Lambda$ to
keep (\ref{eq:us}) valid no matter what initial $x_{1}\in X$ we
choose, we are in no way restricting ourselves by choosing that $x_{1}$
as our ``zero'' instead of a generic $x\in X$ \cite{Pata2011}.
\end{rem}

\subsection*{Proofs}

For the sake of brevity and clarity both, we shall henceforth exploit
the following notation when $j>m$: 
\[
A_{j}\colon=A_{i}\textrm{,}
\]
where $i\equiv j\pmod m$ and $1\leq i\leq m$. 

Let's begin by choosing our zero from $A_{1}$, i.e., we fix $x_{1}\in A_{1}$.
Starting from $x_{1}$, we then introduce the sequence of Picard iterates
\begin{align*}
x_{n} & =Tx_{n-1}=T^{n-1}x_{1} & (n\ge2).
\end{align*}
Also, let 
\begin{align*}
c_{n}\colon & =\left\Vert x_{n}\right\Vert  & (n\in\mathbb{N}).
\end{align*}
With the assumption that $x_{n}\ne x_{n+1}\,\forall n\in\mathbb{N}$,
(\ref{eq:us}) gives us 
\begin{align*}
d(x_{n+1},\, x_{n}) & =d(Tx_{n},\, Tx_{n-1})\\
 & \le\frac{1}{2}[d(x_{n+1},\, x_{n})+d(x_{n},\, x_{n-1})]
\end{align*}
if we consider the case where $\varepsilon=0$. But this means that
\begin{align}
0\leq d(x_{n+1},\, x_{n}) & \le d(x_{n},\, x_{n-1})\nonumber \\
 & \leq\cdots\nonumber \\
 & \le d(x_{2},\, x_{1})\label{eq:md}\\
 & =c_{2},\nonumber 
\end{align}
whence the next result, i.e. our first lemma, is delivered: 
\begin{lem}
\label{lem:bounded}$\{c_{n}\}$ is bounded.\end{lem}
\begin{proof}
Let $n\in\mathbb{N}$. We assume that 
\[
n\equiv k\pmod m\,(1\leq k\leq m).
\]
Since $x_{k-1}\in A_{k-1}$ and $x_{k-2}\in A_{k-2}$, using (\ref{eq:us})
with $\varepsilon=0$, 
\begin{align*}
c_{n} & =d(x_{n},\, x_{1})\\
 & =[d(x_{1},\, x_{2})+d(x_{2},\, x_{3})+\,\cdots\,+d(x_{k-2},\, x_{k-1})]+d(x_{k-1},\, x_{n})\\
 & \leq(k-2)c_{2}+d(Tx_{k-2},\, Tx_{k-1})\\
 & \leq(k-2)c_{2}+\frac{1}{2}[d(x_{k-1},\, x_{k-2})+d(x_{n},\, x_{n-1})]\\
 & \leq(k-2)c_{2}+\frac{1}{2}(c_{2}+c_{2})\\
 & =(k-1)c_{2}.
\end{align*}
And hence we have our proof.\end{proof}
\begin{rem}
One finds in \cite{Alghamdi2012} an attempt to prove the boundedness
of an analogous sequence $c_{n}$ (the notations in play there and
in the present article are virtually the same) using the cyclic contractive
condition from its main theorem (vide inequality (2.1) from theorem
2.4 in \cite{Alghamdi2012}) on two points $x_{1}\,(\in A_{1})$ and
$x_{n}\,(\in A_{n}).$ But this inequality as well as our own (\ref{eq:us})
can only be applied to points that are members of \emph{consecutive}
sets $A_{i}$ and $A_{i+1}$ for some $i\in\{1,\,\ldots\,,\, m\}$
according to their respective applicative restrictions, both of which
stem from the very definition of cyclical conditions given in \cite{kirk2003fixed}.
$x_{n}$ being the general $n$-th term of the sequence $\{x_{n}\}$
is in a \emph{general }set $A_{n}$, and, following the notational
convention agreed upon in both \cite{Alghamdi2012} and this article,
$A_{n}=A_{l}$, where $l\equiv n\pmod m$ and $1\leq l\leq m$. Since
the index $l$ need not either be succeeding or be preceding the index
$1$ in general, $x_{1}\,(\in A_{1})$ and $x_{n}\,(\in A_{l})$ need
not necessarily be members of consecutive sets as well. Hence the
justifiability of using the cyclic criterion on them is lost, and
suitable adjustments have to be made in the structure of the proving
argument. This is precisely what we have endeavoured to do in our
proof above.
\end{rem}
To return to our central domain of discourse, next we need another:
\begin{lem}
\label{lem:limit}$\lim_{n\rightarrow\infty}d(x_{n+1},\, x_{n})=0$.\end{lem}
\begin{proof}
\emph{(\ref{eq:md}) }assures that we end up with a sequence, viz.
$\{d(x_{n+1},\, x_{n})\}$, that is both monotonically decreasing
and bounded below, and, therefore, 
\begin{align*}
\lim_{n\rightarrow\infty}d(x_{n+1},\, x_{n})=\inf_{n\in\mathbb{N}}d(x_{n+1},\, x_{n}) & =r\,\textrm{(say)}\\
 & \geq0.
\end{align*}

But, from (\ref{eq:us}), 
\begin{align*}
r & \leq d(x_{n+1},\, x_{n})\\
 & =d(Tx_{n},\, Tx_{n-1})\\
 & \leq\frac{1-\varepsilon}{2}[d(x_{n+1},\, x_{n})+d(x_{n},\, x_{n-1})]+\Lambda\varepsilon^{\alpha}\psi(\varepsilon)(1+c_{n+1}+2c_{n}+c_{n-1})^{\beta}\\
 & \leq\frac{1-\varepsilon}{2}[d(x_{n+1},\, x_{n})+d(x_{n},\, x_{n-1})]+K\varepsilon\psi(\varepsilon)
\end{align*}
for some $K\geq0$. (By virtue of \lemref{bounded}, it is ensured
that $K$ does not depend on $n$.) Letting $n\rightarrow\infty$,
\begin{align*}
 & r\leq\frac{1-\varepsilon}{2}(r+r)+K\varepsilon\psi(\varepsilon)\\
\implies & r\leq K\psi(\varepsilon) & \forall\varepsilon\in(0,\,1]\\
\implies & r=0.
\end{align*}

Therefore, 
\[
\lim_{n\rightarrow\infty}d(x_{n+1},\, x_{n})=0.
\]

\end{proof}
With this, we are now in a position to derive:
\begin{lem}
\label{lem:Cauchy}$\{x_{n}\}$ is a Cauchy sequence.\end{lem}
\begin{proof}
This proof has the same generic character as the one given in \cite{kirk2003fixed}.
We suppose, first, that $\exists\rho>0$ such that, given any $N\in\mathbb{N}$,
$\exists n>p\geq N$ with $n-p\equiv1\pmod m$ and 
\[
d(x_{n},\, x_{p})\geq\rho>0.
\]
Clearly, $x_{n-1}$ and $x_{p-1}$ lie in different but consecutively
labelled sets $A_{i}$ and $A_{i+1}$ for some $i\in\{1,\,\ldots\,,\, m\}$.
Then, from (\ref{eq:us}), $\forall\varepsilon\in[0,\,1]$, 
\begin{align*}
d(x_{n},\, x_{p}) & \leq\frac{1-\varepsilon}{2}[d(x_{n},\, x_{n-1})+d(x_{p},\, x_{p-1})]+\Lambda\varepsilon^{\alpha}\psi(\varepsilon)(1+c_{n}+c_{n-1}+c_{p}+c_{p-1})^{\beta}\\
 & \leq\frac{1-\varepsilon}{2}[d(x_{n},\, x_{n-1})+d(x_{p},\, x_{p-1})]+C\varepsilon^{\alpha}\psi(\varepsilon),
\end{align*}
where, to be precise, $C=\sup_{j\in\mathbb{N}}\Lambda(1+4c_{j})^{\beta}<\infty$
(on account of \lemref{bounded}) this time. If we let $n,\, p\rightarrow\infty$
with $n-p\equiv1\pmod m$, then \lemref{limit} gives us that 
\[
0<\rho\leq d(x_{n},\, x_{p})\rightarrow0
\]
as $\varepsilon\rightarrow0+$, which is, again, contrary to what
we had supposed earlier.

Therefore, we can safely state that, given $\varepsilon>0$, $\exists N\in\mathbb{N}$
such that 
\begin{equation}
d(x_{n},\, x_{p})\leq\frac{\varepsilon}{m}\label{eq:SortofCauchy}
\end{equation}
whenever $n,\, p\geq N$ and $n-p\equiv1\pmod m$. 

Again, by \lemref{limit} it is possible to choose $M\in\mathbb{N}$
so that 
\[
d(x_{n+1},\, x_{n})\leq\frac{\varepsilon}{m}
\]
if $n\geq M$. If we now let $n,\, p\geq\max\{N,\, M\}$ with $n>p$,
then $\exists r\in\{1,\,2,\,\ldots\,,\, m\}$ such that 
\[
n-p\equiv r\pmod m.
\]
Thus 
\[
n-p+i\equiv1\pmod m,
\]
where $i=m-r+1$. And, bringing into play (\ref{eq:SortofCauchy}),
\begin{align*}
d(x_{n},\, x_{p}) & \leq d(_{p},\, x_{n+i})+[d(x_{n+i},\, x_{n+i-1})+\,\cdots\,+d(x_{n+1},\, x_{n})]\\
 & \leq\varepsilon.
\end{align*}
This proves that $\{x_{n}\}$ is Cauchy.
\end{proof}
Now, looking at $Y=\bigcup_{i}A_{i}$, a complete metric space on
its own, we can conclude straightaway that $\{x_{n}\}$, a Cauchy
sequence in it, converges to a point $y\in Y$. 

But $\{x_{n}\}$ has infinitely many terms in each $A_{i},\, i\in\{1,\,\ldots\,,\, m\}$,
and each $A_{i}$ is a closed subset of $Y$. Therefore, 
\begin{align*}
 & y\in A_{i}\,\forall i\\
\implies & y\in\bigcap_{i=1}^{m}A_{i}\\
\implies & \bigcap_{i=1}^{m}A_{i}\neq\emptyset.\\
\end{align*}
Moreover, $\bigcap_{i=1}^{m}A_{i}$ is, just as well, a complete metric
space per\emph{ }se\emph{.} Thus, considering the restricted mapping
\[
U\colon=T\restriction_{\bigcap A_{i}}\colon\bigcap A_{i}\to\bigcap A_{i},
\]
we notice that it satisfies the criterion to be a Kannan-type generalized
map already proven by us to have a unique fixed point $x^{*}\in\bigcap A_{i}$
by virtue of \thmref{our-1st}. \qed 
\begin{rem}
We have to minutely peruse a certain nuance here for rigour's sake:
the moment we know that $\bigcap A_{i}\neq\emptyset$, we can choose
an arbitrary $y_{1}\in\bigcap A_{i}$ to serve as its zero, and the
restriction of $T$ to $\bigcap A_{i}$ can still be made to satisfy
(a modified form of) (\ref{eq:us}) insofar as $\Lambda$ can be appropriately
revised as per remark \remref{Since-we-can}; this renders the employment
of \thmref{our-1st} in the proof vindicated.
\end{rem}

\section{Some Conclusions}

Following the terminology of \cite{rus2003picard}, we can actually
show something more, viz.:
\begin{cor}
$T$ is a Picard operator, i.e. $T$ has a unique fixed point $x^{*}\in\bigcap_{i=1}^{m}A_{i}$,
and the sequence of Picard iterates $\{T^{n}x\}_{n\in\mathbb{N}}$
converges to $x^{*}$ irrespective of our initial choice of $x\in Y$. \end{cor}
\begin{proof}
So far, we have already shown that for a fixed $x_{1}\in A_{1}$ \emph{one
and only one} fixed point $x^{*}$ of $T$ exists. To complete the
proof, let's first observe that the decision to let $x_{1}\in A_{1}$
at the beginning of the main proof was based partly on mere convention
and partly on an intention to develop our argument thenceforth more
or less analogously to the proof given in \cite{Alghamdi2012}; if
we would have chosen any generic $x\in Y$ instead, then, seeing as
how $Y=\bigcup A_{i}$, that $x$ would have belonged to $A_{j}$
for some $j\in\{1,\,\ldots\,,\, m\}$, and our discussion thereupon
would have differed only in some labellings, \emph{not} in its conclusion:
i.e. we would have, eventually, inferred the existence of a unique
fixed point of $T$ in $\bigcap A_{i}$. 

Next we want to demonstrate that the convergence of the Picard iterates
heads to a fixed point of $T$. To this end we recall that $\{T^{n}x_{1}\}_{n\in\mathbb{N}}=\{x_{n+1}\}_{n\in\mathbb{N}}$
converges to $y\in\bigcap A_{i}$. Our claim is that this $y$ itself
is a fixed point of $T$. This we can verify summarily: 

As $x_{n}\in A_{k}$ for some $k\in\{1,\,\ldots\,,\, m\}$ and as
$y\in\bigcap_{i=1}^{m}A_{i}\subset A_{k+1}$, 

\begin{align*}
d(y,\, Ty) & \leq d(y,\, x_{n+1})+d(x_{n+1},\, Ty)\\
 & =d(y,\, x_{n+1})+d(Tx_{n},\, Ty)\\
 & \leq d(y,\, x_{n+1})+\frac{1}{2}[d(x_{n+1},\, x_{n})+d(y,\, Ty)]
\end{align*}
for every $n\in\mathbb{N}$, using (\ref{eq:us}) with $\varepsilon=0$
again, and, from that,
\[
\frac{1}{2}d(y,\, Ty)\leq d(y,\, x_{n+1})+\frac{1}{2}d(x_{n+1},\, x_{n})
\]
for all $n$. Letting $n\rightarrow\infty$, 
\[
d(y,\, Ty)=0
\]
since $x_{n+1}\rightarrow y$ and $d(x_{n+1},\, x_{n})\rightarrow0$.
Therefore,
\[
y=Ty.
\]

As observed, the choice of the starting point $x_{1}$ is irrelevant,
and we already know that $x^{*}\in\bigcap A_{i}$ is \emph{the} unique
fixed point of $T$. So obviously,
\[
x^{*}=y,
\]
i.e. $T^{n}x_{1}\rightarrow x^{*}$ as $n\rightarrow\infty$. 

This completes the proof.\end{proof}
\begin{rem}
Trying to show their map $f$ (corresponding to the $T$ in the present
article) is a Picard operator, the authors in \cite{Alghamdi2012}
have set out to prove that $x_{n}\rightarrow x^{*}$ as $n\rightarrow\infty$
too. (Again, apart from denoting the operator differently, we haven't
really changed much of their notation so as to make a comparison fairly
self-explanatory.) The argument they have used in this follows from
a technique given in \cite{Pata2011}, and they, too, have ended up
stating, to quote a portion of the concerned reasoning in \cite{Alghamdi2012}
verbatim,
\[
d(x_{n},\, x^{*})=\lim_{p\rightarrow\infty}d(x_{n},\, x_{n+p}).
\]
This is where a problem arises. 

In \cite{Pata2011} the convergence of the Cauchy sequence $x_{n}$
is established directly to its limit $x_{*}$, and, therefore, the
utilization of an equality like
\[
d(x_{*},\, x_{n})=\lim_{m\rightarrow\infty}d(x_{n+m},\, x_{n})
\]
(quoted as it is from \cite{Pata2011} this time) is perfectly justified.
The cyclical setting in both this article and \cite{Alghamdi2012},
however, only ensures initially that $x_{n}\rightarrow y$ for some
$y\in\bigcap A_{i}$ as $n\rightarrow\infty$, and, as a consequence,
guarantees next the existence of a unique fixed point $x^{*}$ for
the operator. The fact that this $y$ turns out be a fixed point for
the operator as well (thereby rendering it equal to the unique $x^{*}$)
is something that needs to be \emph{actually proved} in a separate
treatment, which we believe is a task we've accomplished. \cite{Alghamdi2012},
though, overlooks this distinction and assumes the very fact (viz.
$x_{n}\rightarrow x^{*}$) it wants to prove in the proof itself,
committing the fallacy of \emph{petitio principii}.
\end{rem}
As a final note, let us remind ourselves of the fact that (\ref{eq:our-1st})
is weaker than (\ref{eq:kannan}) (see \cite{Chakraborty2012first}),
and, in light of this we also have, as a corollary to our \thmref{main}
the following:
\begin{cor}
\cite{rus2005cyclic,petric2010fixed} Let $\{A_{i}\}_{i=1}^{p}$ be
nonempty closed subsets of a complete metric space $X$. Suppose that
$T\colon\bigcup_{i=1}^{p}A_{i}\to\bigcup_{i=1}^{p}A_{i}$ is a cyclic
map, i.e. it satisfies $T(A_{i})\subset A_{i+1}$ for every $i\in\{1,\,\ldots\,,\, p\}$
(with $A_{p+1}=A_{1}$), such that
\[
d(Tx,\, Ty)\leq\frac{\alpha}{2}[d(x,\, Tx)+d(y,\, Ty)]
\]
for all $x\in A_{i}$, $y\in A_{i+1}$ ($1\leq i\leq p$), where $\alpha\in(0,\,1)$
is a constant. Then $T$ has a unique fixed point $x^{*}$ in $\bigcap_{i=1}^{p}A_{i}$
and is a Picard operator.
\end{cor}
\bibliographystyle{amsplain}
\bibliography{references}

\end{document}